\documentclass[12pt]{article}

\usepackage{datetime}

\usepackage[all]{xy}

\usepackage{amssymb, amsmath, amsthm,eucal}

\usepackage{mathptmx}
\usepackage[scaled=.90]{helvet}
\usepackage{courier}

\usepackage[pdftex]{hyperref}

\hypersetup{
  colorlinks = true,
  linkcolor = black,
  urlcolor = blue,
  citecolor = black,
}

\numberwithin{equation}{section}

\renewcommand{\phi}{\varphi}
\renewcommand{\O}{\mathcal{O}}
\renewcommand{\H}{\operatorname{H}}

\newcommand{\Het}{\H_{\mathrm{et}}}

\newcommand{\Pic}{\mathrm{Pic}}
\newcommand{\fPic}{\mathrm{\widehat Pic}}
\newcommand{\Br}{\mathrm{Br}}
\newcommand{\fBr}{\mathrm{\widehat Br}}

\newcommand{\B}{\mathrm{B}}
\newcommand{\Lie}{\operatorname{Lie}}
\newcommand{\Spec}{\operatorname{Spec}}

\newcommand{\GG}{\mathbb{G}}
\newcommand{\PP}{\mathbb{P}}
\newcommand{\QQ}{\mathbb{Q}}
\newcommand{\TT}{\mathbb{T}}
\newcommand{\ZZ}{\mathbb{Z}}

\newcommand{\M}{\mathcal{M}}
\newcommand{\Mo}{\M^{\mathrm{fin}}}

\newtheorem{theo}{Theorem}
\newtheorem{lemm}[theo]{Lemma}
\newtheorem{prop}[theo]{Proposition}

\theoremstyle{definition}

\newtheorem*{exam*}{Example}

\newtheorem*{defi*}{Definition}

\title{\bf~K3 spectra}

\author{Markus Szymik}

\date{January 2009}

\begin{document}

\maketitle

\begin{abstract}
\noindent
  The notion of a~K3 spectrum is introduced in analogy with that of an
  elliptic spectrum and it is shown that there are ``enough''~K3
  spectra in the sense that for all~K3 surfaces~$X$ in a suitable
  moduli stack of~K3 surfaces there is a~K3 spectrum whose underlying
  ring is isomorphic to the local ring of the moduli stack in~$X$ with
  respect to the etale topology, and similarly for the ring of formal
  functions on the formal deformation space.
\end{abstract}

\thispagestyle{empty}

%%%%%%%%%%%%%%%%%%%%%%%%%%%%%%%%%%%%%%%%%%%%%%%%%%%%%%%%%%%%%%%%%%%%%

\section{Introduction}

Stable homotopy theory may be defined as the theory of the sphere
spectrum~$S$ and its homotopy groups~$\pi_*S$, the stable homotopy
groups of spheres. But, just as number theory may be defined as the
theory of the integers, in practice that entails the study of other
structures such as fields, both global and local, Galois groups, their
representations, deformations, and many other things which --
typically -- are harder to construct, but easier to understand.

In stable homotopy theory, some of the auxiliary spectra encountered
are the Eilenberg-MacLane spectra such as~$\mathrm{H}\QQ$, leading to
rational homotopy theory, topological K-theory spectra~$\mathrm{KU}$
and~$\mathrm{KO}$, elliptic spectra and~$\mathrm{TMF}$, and many
more. The ``chromatic'' hierarchy of these is organised by the complex
bordism spectrum~$\mathrm{MU}$, which is relevant for the fact
that~$\pi_*\mathrm{MU}$ represents graded formal group laws rather
than for its relationship with manifolds. For example,~$\mathrm{H}\QQ$
corresponds to the additive formal group law,~$\mathrm{KU}$ to the
multiplicative formal group law, and elliptic spectra to formal group
laws coming from elliptic curves.

There have been efforts to extend the connection between elliptic
curves and spectra to other geometric objects such as curves of higher
genus, see~\cite{Gorbounov+Mahowald} and~\cite{Ravenel}, and abelian
varieties of higher dimension, see~\cite{Behrens+Lawson}. The aim of
this report is to present some other part of arithmetic geometry which
may be lurking behind the next layers of the chromatic hierarchy:~K3
surfaces and their corresponding~K3 spectra.

It is not a new idea to use~K3 surfaces, and more generally
Calabi-Yau varieties, as generalisations of elliptic curves, not
even in topology. I first read about the idea of cohomology theories
related to~K3 surfaces in Thomas' writings~\cite{Thomas:Book}
and~\cite{Thomas:Survey}, where he refers to Morava. The latter was so
kind to share his notes~\cite{Morava} based then again on a lecture of
Hopkins. See also~\cite{Devoto} for a recent contribution to the
problem of finding a differential geometric (rather than homotopical)
description of~K3 cohomology.

In recent years a lot of work has been done, on the side of algebraic
topology just as well as on the side of arithmetic geometry, and
suggests to have a fresh look at this connection. For example, the
papers~\cite{vanderGeerKatsura:Stratification}
and~\cite{Ogus:HeightStrata} give detailed accounts of the height
stratification on the moduli stacks of~K3 surfaces in prime
characteristic. And there are now various different methods to produce
new species in the ``brave new world'' of highly structured ring
spectra and localisations thereof, see~\cite{Rezk:HopkinsMiller},
\cite{GoerssHopkins:Summary}, and~\cite{Lurie:Elliptic} to name a
few. How far these apply to~K3 spectra is the author's work in
progress and will be addressed elsewhere~\cite{Szymik}. 

The purpose here is to show that there are ``enough''~K3 spectra in
the sense that for all~K3 surfaces~$X$ in a suitable moduli stack of
K3 surfaces there is a~K3 spectrum whose underlying ring is isomorphic
to the local ring of the moduli stack in~$X$ with respect to the
etale topology, and similarly for the ring of formal functions on the
formal deformation space. The precise statements are given in
Propositions~\ref{prop:local} and~\ref{prop:formal}, which follow from
the main Theorem~\ref{thm:main}. The proofs for these are in
Section~5. Sections 2 and 3 review the geometry and arithmetic of~K3
surfaces, respectively, and Section 4 gives the definition of~K3
spectra and examples over the rationals.

%%%%%%%%%%%%%%%%%%%%%%%%%%%%%%%%%%%%%%%%%%%%%%%%%%%%%%%%%%%%%%%%%%%%%

\section{K3 surfaces}

The aim of this section is to review the geometry of~K3 surfaces as
far as needed in the rest of the text. All fields will be assumed
perfect from now on.

\subsection{Definition of~K3 surfaces}

An elliptic curve over a field~$k$ is a smooth proper curve~$X$ such
that the canonical bundle~\hbox{$\omega_X=\bigwedge^1(\Omega_X)$} is
trivial, with a chosen base point. The base point can be used for
various purposes. It can be used to impose an abelian group structure
on the curve. And it can be used to define a Weierstrass embedding
into the projective plane.

One dimension higher, there are two kinds of smooth proper surfaces
$X$ such that the canonical bundle~$\omega_X=\bigwedge^2(\Omega_X)$ is
trivial: abelian surfaces and another class of surfaces which satisfy
the additional condition~$\H^1(X,\O_X)=0$: the {\it~K3 surfaces}.  See
the textbooks~\cite{Beauville}, Chapter~VIII, and~\cite{Badescu},
Chapter 10, for the general theory of~K3 surfaces.

One generalisation of this to even higher dimensions would be
Calabi-Yau~$n$-folds, which are defined by the triviality of the
canonical bundle~\hbox{$\omega_X=\bigwedge^n(\Omega_X)$} and the
vanishing of~$\H^i(X,\O_X)$ for all~$1\leqslant i\leqslant
n-1$. See~\cite{Hirokado} and~\cite{vanderGeerKatsura:Heights}, for
example. Apart from a few side remarks, these will play no r\^ole in
the following.

The Euler characteristic of~K3 surfaces is 24, so that these will
never be abelian groups. However, every~K3 surface~$X$ can be embedded
in some projective space by means of an ample line bundle~$L$ on~$X$;
the choice of such an~$L$ is called a {\it polarisation} of~$X$. This
corresponds to the choice of the base point for elliptic curves. The
self-intersection of~$L$ is called the {\it degree} of the
polarisation; this will always be an even integer~$2d$ for some
$d\geqslant1$.

\subsection{Examples of~K3 surfaces}

It will be useful to bear the following two classes of examples in
mind. For these, the base field~$k$ needs to be of
characteristic~\hbox{$\operatorname{char}(k)\not=2$}.

\begin{exam*}
  The most famous example for a~K3 surface is the {\it Fermat quartic}
  defined by the equation
  \begin{displaymath}
    T_0^4+T_1^4+T_2^4+T_3^4=0
  \end{displaymath}
  in projective 3-space~$\PP^3$. More generally, all smooth quartics
  in~$\PP^3$ define~K3 surfaces.  (See~\cite{Beauville}, Example
  VIII.8, and~10.4 in~\cite{Badescu}.)  This is analogous to the fact
  that smooth cubics in~$\PP^2$ define elliptic curves. However, not
  every~K3 surface can be embedded into~$\PP^3$.
\end{exam*}

\begin{exam*}
  If~$A$ is an abelian surface, the inversion~$[-1]\colon A\rightarrow
  A$ is an involution, and it extends to the blow-up of~$A$ along the
  sixteen 2-torsion points. The quotient of the resulting free action
  on the blow-up by this involution turns out to be a~K3 surface,
  called the {\it Kummer surface} of~$A$. See~\cite{Beauville},
  Example VIII.10, and~10.5 in~\cite{Badescu}.
\end{exam*}

\subsection{Moduli of~K3 surfaces}\label{subsec:moduli}

Given the two classes of examples above, one might want to get an
overview of all~K3 surfaces and how they vary in families, leading to
the question of their moduli. While the moduli stack of elliptic
curves is only 1-dimensional, the moduli stack of Kummer surfaces
is~3-dimensional, and that of quartics in projective space
is~19-dimensional. In general, one may consider the moduli stack
$\M_{2d}$ of polarised~K3 surfaces of fixed
degree~$2d$. See~\cite{Rizov:MixedCharacteristic}. That is a separated
Deligne-Mumford stack of finite type, which is smooth of dimension~19
over~$\ZZ[1/2d]$. The objects of~$\M_{2d}$ should be the pairs
$(X,L)$, where~$X$ is a~K3 surface over some~$\ZZ[1/2d]$-scheme~$S$,
and~$L$ is a polarisation for~$X$; morphisms~\hbox{$(X',L')\rightarrow(X,L)$} are pullback diagrams
\begin{center}
  \mbox{ 
    \xymatrix{
      X'\ar[r]_{\xi}\ar[d]_{}&X\ar[d]^{} \\
      S'\ar[r]^{\sigma}&S
    } 
  }
\end{center}
such that~$\xi^*L=L'$. But, in order to circumvent difficulties with
etale descent, one builds this into the definition of~$\M_{2d}$ by
allowing for objects which become~K3 surfaces only after an etale base
change. Again, see~\cite{Rizov:MixedCharacteristic} for details. The
difference is clearly irrelevant over the spectrum of a strictly
henselian local ring.

\begin{exam*} 
  The case~$d=2$ corresponds to the quartics
  in~$\PP^3$. (See~\cite{Mukai} for other small~$d$.) The space of
  quartics in~$\PP^3$ is a projective space of dimension 
  \begin{displaymath} 
    \dim\H^0(\PP^3;\O_{\PP^3}(4))-1=\binom{7}{4}-1=34.
  \end{displaymath}
  The discriminant locus~$\Delta\subset\PP^{34}$ corresponding to the
  singular quartics is a hyper\-surface. The Veronese embedding shows
  that the open complement~$\PP^{34}\setminus\Delta$ is affine. The
  universal~K3 surface over it has a canonical polarisation given by
  the restriction of~$\O_{\PP^3}(1)$. There results a morphism
  \begin{displaymath}
    \PP^{34}\setminus\Delta\longrightarrow\M_{4}.
  \end{displaymath}
  This is not surjective, as ample line bundles on~K3 surfaces need not
  be very ample; the canonical image in~$\PP^3$ may acquire ordinary
  double points. However, the map factors over the quotient stack
  \begin{displaymath}
    \left(\PP^{34}\setminus\Delta\right)/\!\!/\,\mathrm{PGL}(4)
  \end{displaymath}
  of the affine scheme~$\PP^{34}\setminus\Delta$ by the action of the
  affine group~$\mathrm{PGL}(4)$ which acts by change of co-ordinates,
  and links~$\M_{4}$ to a stack associated to a Hopf algebroid.
\end{exam*}

Finally, it should be pointed out that the moduli problem of
(polarised) Calabi-Yau~$n$-folds is much more complicated if
$n\geqslant3$. There are~$3$-folds with different Hodge numbers, so
that~-- globally~-- there will be many components. The Hodge numbers
also show that~-- locally~-- the obstruction space for deformations
need not be zero (as it is for~K3 surfaces). And in fact there are
examples of Calabi-Yau~3-folds which do not lift to characteristic
zero, see~\cite{Hirokado}.

%%%%%%%%%%%%%%%%%%%%%%%%%%%%%%%%%%%%%%%%%%%%%%%%%%%%%%%%%%%%%%%%%%%%%

\section{Formal Brauer groups}

The aim of this section is to review the arithmetic of~K3 surfaces as
far as needed in the rest of the text.

\subsection{Formal groups}

Let~$R$ be a local ring with residue field~$k$. A functor~$\Gamma$
from the category of artinian local~$R$-algebras with residue field
$k$ to the category of abelian groups is a (1-dimensional,
commutative) {\it formal group} if the underlying functor to the
category of sets is (pro-)representable, formally smooth, and
1-dimensional. However, the actual representation is not part of the
data. The choice of a representation (necessarily by a power series
ring~$R[\![T]\!]$) yields a {\it co-ordinate} for~$\Gamma$. With
respect to this co-ordinate the formal group is described by the {\it
  formal group law}, which is a power series in
$R[\![T_1,T_2]\!]$. See the lecture notes~\cite{Fröhlich}
and~\cite{Lazard:LectureNotes} for background information. There are
ways to globalise the notion of a formal group in order to work over
arbitrary rings or even base schemes. In that case, co-ordinates need
not exist globally, only locally. This generality will not be needed
here.

\subsection{The theorem of Artin and Mazur}

Let~$X$ be an elliptic curve over a field~$k$. Its Picard group
$\Pic_X(k)$ is the group of isomorphism classes of line bundles on
$X$. As it stands, this is just a set with an abelian group structure,
but with some work it can be made into an algebraic group~$\Pic_X$,
whose component of the identity is isomorphic to~$X$ by a map which
sends the base point to the unit. The formal completion~$\fPic_X$ is
a~1-dimensional formal group. It has been those formal groups that
arose the interest of algebraic topologists in elliptic curves,
see~\cite{LRS} for example.

For a~K3 surface~$X$, the Picard group is 0-dimensional, so that its
formal Picard group is trivial, but the isomorphism
\begin{displaymath}
	\Pic_X(k)\cong\Het^1(X;\GG_m) 
\end{displaymath}
suggests that one should consider~$\Het^2(X;\GG_m)$ instead. As for a
geometric interpretation, there is an isomorphism
\begin{displaymath}
	\Het^2(X;\GG_m)\cong\Br_X(k)
\end{displaymath}
with the Brauer group~$\Br_X(k)$ of~$X$. Although the notation may
suggest that, it turns out that these are not the~$k$-valued points of
an algebraic group, let alone one whose component of the identity is
isomorphic to~$X$. However, Artin and Mazur~\cite{Artin+Mazur} have
shown that the functor
\begin{displaymath}
  A\longmapsto
  \operatorname{Ker}\left(
    \Het^2(X\times\Spec(A);\GG_m)\rightarrow\Het^2(X;\GG_m)
  \right)
\end{displaymath}
on artinian~$k$-algebras~$A$ with residue field~$k$ is (pro-)representable
and formally smooth. This is referred to as the {\it formal Brauer
  group}~$\fBr_X$ of~$X$. Its dimension is~1; in fact
\begin{equation}\label{eq:LieBrauer}
	\Lie(\fBr_X)\cong\Het^2(X;\GG_a)\cong\H^2(X,\O_X)
\end{equation}
gives the Lie algebra. These formal groups are the reason for the
topologists' interest in~K3 surfaces.

As a remark, the results of Artin and Mazur are general enough to show
that every Calabi-Yau~$n$-fold gives rise to a 1-dimensional formal
group, using the same construction based on~$\Het^n(X;\GG_m)$.

\subsection{Heights}\label{subsec:Fermat}

Over rings of prime characteristic, formal groups have an associated
height, which can be a positive integer or infinite. It counts the
maximal number of Frobenius morphisms over which multiplication by~$p$
factors. It is the most basic invariant of formal groups in prime
characteristic, and over separably closed fields it is in fact their
only invariant. See~\cite{Lazard:Classification} and~\cite{Fröhlich},~III.2. 
Over local rings, the height is dominated by the height over
the residue field.
 
The multiplicative formal group~$\hat\GG_m$ has height~1, whereas the
additive formal group~$\hat\GG_a$ has infinite height. The height of
the formal Picard group of an elliptic curve is either 1 or 2, in
which case the elliptic curve is called ordinary or supersingular,
respectively.

\begin{exam*}
  It is known that the height of the formal Brauer group of the Fermat
  surface is 1 in the case when~$p\equiv 1$ modulo~$4$ and infinite
  if~$p\equiv 3$ modulo~$4$, see page 91 in~\cite{Artin+Mazur} and
  page 543 in~\cite{Artin}, respectively. The inclined reader may want
  to see this by means of Stienstra's formula
  \begin{displaymath}
    \log'(T)=\sum_{n=0}^\infty\frac{(4n)!}{(n!)^4}T^{4n}
  \end{displaymath}
  for the derivative of the logarithm~\cite{Stienstra:fgl}. This example
  shows that the height need not vary nicely from prime to prime in a
  family of mixed characteristic.
\end{exam*}

\begin{exam*}
  The height of formal Brauer groups of Kummer surfaces are 1, 2 or
  infinite, see~\cite{vanderGeerKatsura:Heights}. Specifically, to get
  height 2, take the Kummer surface of a product of an ordinary and a
  supersingular elliptic curve.
\end{exam*}

In general, the height of the formal Brauer group of a~K3 surface will
be at most 10 or infinite, and all these possibilities actually
occur. There are explicit examples known for all heights except for 7,
see~\cite{Yui} and~\cite{Goto}. A~K3 surface is called {\it ordinary}
or {\it supersingular} (in the sense of Artin, see~\cite{Artin}) if
the height is 1 or infinite, respectively. 

The pattern continues: there is a bound on the height of the
Artin-Mazur formal group associated to a Calabi-Yau~$n$-fold,
see~\cite{vanderGeerKatsura:Heights} again, but note that the bound
given there is not sharp even in the case of~K3 surfaces.

Now let~$p$ be a prime that does not divide~$2d$, and let
\begin{displaymath}
	\M_{2d,p}
\end{displaymath}
be the base change of~$\M_{2d}$ from~$\ZZ[1/2d]$ to~$\ZZ/p$. For all
$h\geqslant 1$, there is a closed substack
\begin{displaymath}
	\M_{2d,p,h} 
\end{displaymath}
of~$\M_{2d,p}$ defined by those~K3 surfaces which have a formal Brauer
group of height at least~$h$. These define the {\it height
  stratification}
\begin{equation}\label{height stratification}
  \M_{2d,p}=\M_{2d,p,1}
  \supseteq\M_{2d,p,2}
  \supseteq\M_{2d,p,3}\supseteq\dots
\end{equation}
of~$\M_{2d,p}$, see~\cite{vanderGeerKatsura:Stratification}
and~\cite{Ogus:HeightStrata}. By what has been said above, the chain
stabilises at~$h=11$, with~$\M_{2d,p,11}$ being the supersingular
locus, at least set-theoretically. Its open complement
\begin{displaymath}
  \Mo_{2d,p,h} 
\end{displaymath}
in~$\M_{2d,p,h}$ is the moduli stack of polarised~K3 surfaces of
finite height at least~$h$ in characteristic~$p$. These are known to
be smooth of dimension~$20-h$ for~\hbox{$h=1,\dots,10$}, and empty
for~$h=11$. The substack~$\Mo_{2d,p,h+1}$ of~$\Mo_{2d,p,h}$ is defined
by the vanishing of a section of an invertible sheaf of ideals, see
again~\cite{vanderGeerKatsura:Stratification}
and~\cite{Ogus:HeightStrata}. It follows that these sections define a
regular sequence locally on the moduli stack.

%%%%%%%%%%%%%%%%%%%%%%%%%%%%%%%%%%%%%%%%%%%%%%%%%%%%%%%%%%%%%%%%%%%%%

\section{K3 spectra}

The aim of this section is to define the notion of a~K3 spectrum in
analogy with elliptic spectra.

\subsection{Even periodic ring spectra}
 
Let~$E$ be a ring spectrum in the weak ``up to homotopy'' sense. In
other words~$E$ is just a monoid with respect to the smash product in
the homotopy category of spectra. Nowadays, there are various models
for the category of spectra which come with a well-behaved smash
product before passage to the homotopy category, so that it also makes
sense to study monoids in a category of spectra
itself. See~\cite{MMSS} for a comparison of some of the most common
models. In this writing, the stronger notion will not be discussed.

The ring spectrum~$E$ is called {\it even} if its homotopy groups are
trivial in odd degrees, and {\it periodic} if the multiplication
induces isomorphisms
\begin{displaymath}
  \pi_{2m}E\otimes_{\displaystyle \pi_0E}\pi_{2n}E
  \stackrel{\cong}{\longrightarrow}
  \pi_{2(m+n)}E
\end{displaymath}
for all integers~$m$ and~$n$. An even periodic ring spectrum has an
associated formal group~$\Gamma_E$, represented by the ring
$\pi_0E^{\B\TT}=E^0\B\TT$. This ring will be interpreted as the ring
of functions on the formal group over~$\pi_0E$.

The projection from~$\B\TT$ to a point induces the structure map
$\pi_0E\rightarrow\pi_0E^{\B\TT}$ of the~$\pi_0E$-algebra. The
unit~\hbox{$*=\B1\rightarrow\B\TT$} of the group~$\TT$ induces a
co-unit~\hbox{$\pi_0E^{\B\TT}\rightarrow\pi_0E$}. This is to be
interpreted as the evaluation map at the origin. Its kernel~$I$ is the
ideal of functions vanishing at the origin. The cotangent
space~\hbox{$I/I^2$} at the origin can then be identified with
$\pi_2E$, so that its dual~$\pi_{-2}E$ is the Lie algebra
\begin{equation}\label{eq:LieGamma}
  \Lie(\Gamma_E)\cong\pi_{-2}E
\end{equation}
of~$\Gamma_E$.

\subsection{Definition of~K3 spectra}

Recall, or see~\cite{Hopkins:ICM1994} for example, that an elliptic
spectrum is a triple~$(E,X,\phi)$ consisting of an even periodic ring
spectrum~$E$, an elliptic curve~$X$ over~$\pi_0E$, and an
isomorphism~$\phi$ of the formal Picard group of~$X$ with the formal
group~$\Gamma_E$ associated to~$E$ over~$\pi_0E$. There may be reasons
to allow~$X$ to be some sort of generalised elliptic curve, but these
will not matter here.

\begin{defi*} 
  A {\it~K3 spectrum} is a triple~$(E,X,\phi)$ consisting of an even
  periodic ring spectrum~$E$, a~K3 surface~$X$ over~$\pi_0E$, and an
  isomorphism~$\phi$ of the formal Brauer group of~$X$ with the formal
  group~$\Gamma_E$ of~$E$.
\end{defi*}

The purpose of the rest of this text is to show how one obtains
examples of~K3 spectra.

\subsection{Examples of~K3 spectra}\label{subsec:rational}

Let~$X$ be a~K3 surface over a field of characteristic 0, so that the
formal Brauer group of the~K3 surface~$X$ is automatically
additive. By~\eqref{eq:LieBrauer}, its Lie algebra is~\hbox{$\Lie(\fBr_X)\cong\H^2(X;\GG_m)$}, and the usual logarithm
$\GG_m\cong\GG_a$ induces an
isomorphism~\hbox{$\H^2(X;\GG_m)\cong\H^2(X;\GG_a)$}. The latter group is
just~$\H^2(X;\O_X)$ which calculates the Lie algebra of~$\fBr_X$:
\begin{displaymath}
	\Lie(\fBr_X)\cong\H^2(X;\O_X).
\end{displaymath}
Let~$E$ be the Eilenberg-MacLane spectrum for the (graded) canonical
ring of~$X$, so that
\begin{displaymath}
	\pi_{2n}E=\H^0(X;\omega_X^{\otimes n}).
\end{displaymath}
As has been noted above, see~\eqref{eq:LieGamma}, the Lie algebra of
$\Gamma_E$ is~$\pi_{-2}E$, which is the dual of~$\pi_2E$:
\begin{displaymath}
	\Lie(\Gamma_E)\cong\H^0(X;\omega_X)^{\vee}.
\end{displaymath}
Let~$\phi$ be the unique isomorphism between~$\fBr_X$ and~$\Gamma_E$
such that the induced isomorphism on Lie algebras is Serre duality
\begin{displaymath}
	\H^2(X;\O_X)\cong\H^0(X;\omega_X)^{\vee}.
\end{displaymath}
Then~$(E,X,\phi)$ is a~K3 spectrum.

The next section will explain how to obtain examples in non-trivial
characteristic.

%%%%%%%%%%%%%%%%%%%%%%%%%%%%%%%%%%%%%%%%%%%%%%%%%%%%%%%%%%%%%%%%%%%%%

\section{Landweber exactness}

The aim of this section is to show that there are ``enough''~K3
spectra. Their existence will be a consequence of the regularity of
the height stratification and Landweber's exact functor theorem, as
suggested in~\cite{Morava}. However, in view of later applications to
highly structured~K3 spectra, extra care will be taken to work over
torsion-free rings instead of rings of characteristic~$p$.

\subsection{The exact functor theorem}

Let us first recall Landweber's theorem
from~\cite{Landweber:Exactness}. See also~\cite{Miller:Landweber}.

Let~$\Gamma$ be a formal group over a local ring~$R$. Choose a
co-ordinate~$T$ and let
\begin{displaymath}
	[p](T)=a_0T+a_1T^2+\dots+a_{p-1}T^p+a_pT^{p+1}+\dots
\end{displaymath}
be the~$p$-series of~$\Gamma$ with respect to that co-ordinate. For an
integer~$n\geqslant0$ let~$I_{p,n}$ be the ideal generated by the
first~$p^{n-1}$ coefficients. It is known that this does not depend
on the co-ordinate. There results an increasing sequence
\begin{equation}\label{eq:Landweber}
	0=I_{p,0}\subseteq I_{p,1}\subseteq I_{p,2}\subseteq\dots
\end{equation}
of ideals. Note that~$I_{p,1}$ is the ideal~$(p)$ generated by
$a_0=p$. For~$n\geqslant 1$, the surjection~\hbox{$R/p\rightarrow
  R/I_{p,n}$} corresponds to the closed subscheme of~$\Spec(R/p)$
where the height of~$\Gamma$ is at least~$n$. It is also known that
there is a sequence~$(v_n\:|\:n\geqslant0)$ of elements in~$R$ such
that~$I_{p,n+1}$ is generated by~$I_{n,p}$ and~$v_n$. The formal
group~$\Gamma$ is called {\it regular at~$p$} if
$(v_n\:|\:n\geqslant0)$ is a regular sequence in~$R$. This does not
depend on the choice of the~$v_n$. For example, if~$p$ is invertible
in~$R$, then~$\Gamma$ is automatically~$p$-regular.

The graded formal group law over the graded ring~$R[u^{\pm1}]$ which
is defined by~$\Gamma$ is classified by a graded
morphism~\hbox{$\mathrm{MU}_*\rightarrow R[u^{\pm1}]$}.  Landweber's
theorem states that the functor
\begin{displaymath}
	X\mapsto \mathrm{MU}_*X\otimes_{\mathrm{MU}_*}^\Gamma R[u^{\pm1}]
\end{displaymath}
is a homology theory if~$\Gamma$ is~$p$-regular for all primes~$p$ and
the sequence~(\ref{eq:Landweber}) eventually stabilises. This homology
theory is representable by an even periodic ring spectrum~$E$ which
has~\hbox{$\pi_0E\cong R$} and~\hbox{$\Gamma_E\cong\Gamma$}.

\subsection{The statement}

As before, let us fix an integer~$d\geqslant 1$, and consider the
moduli stack~$\M_{2d}$ of polarized~K3 surfaces of degree~$2d$ over
$\ZZ[1/2d]$. Suppose that~$X$ is a~K3 surface over an affine scheme
$\Spec(R)$ on which~$2d$ is invertible. Then~$X$ is classified by a
map
\begin{displaymath}
	X\colon\Spec(R)\longrightarrow\M_{2d}.
\end{displaymath}
The following theorem will assert that -- under certain conditions --
the associated formal Brauer group is Landweber exact. Before giving
the precise statement, let me motivate the choice of hypotheses.

First of all, flatness of the map classifying~$X$ will be
indispensable for the argument. This will imply that~$R$ is flat over
$\ZZ[1/2d]$, so that~$R$ is torsion-free, as required by Landweber's
theorem. The primes which divide~$2d$ are automatically units
in~$R$. If~$R$ is a~$\QQ$-algebra, there is no need to invoke
Landweber's theorem to get examples, see the previous
Section~\ref{subsec:rational}. On the other hand, the height
stratification does not vary nicely from prime to prime, see
Section~\ref{subsec:Fermat}. Therefore, we will concentrate on one
prime and assume that the ring~$R$ is a local~$\ZZ_{(p)}$-algebra for
some prime~$p$ which does not divide~$2d$. Recall that ``local'' means
that~$R$ has a unique maximal ideal~$m$, and that this maximal ideal
contains~$p$, so that the residue characteristic of~$R$ is~$p$. There
may be other settings which make the following argument work, but this
one has its merits, as will hopefully become clear in due course.

\begin{theo}\label{thm:main}
  Let~$R$ be a noetherian local~$\ZZ_{(p)}$-algebra for some 
  prime~$p$ which does not divide~$2d$. Let~$X$ be a polarised~K3 surface of degree~$2d$ over~$R$ such that the height of 
  the closed fibre is finite. If the map
  \begin{equation}\label{eq:classifying X}
    X\colon\Spec(R)\longrightarrow\M_{2d}
  \end{equation}
  classifying~$X$ is flat, then the formal Brauer group~$\fBr_X$ 
  is Landweber exact, so that there is an even periodic ring 
  spectrum~$E$ with~$\pi_0E\cong R$ and~\hbox{$\Gamma_E\cong\fBr_X$}.
\end{theo}

Examples will be given after the proof, showing how this can be
used to show that there are ``enough''~K3 spectra.

\subsection{The proof}

By assumption on the ring~$R$, all primes different from~$p$ are
invertible, so that the formal group will automatically be~$q$-regular
for all~$q\not=p$. It remains to show that it is~$p$-regular as well.

\begin{lemm}\label{lem:p}
  The prime~$p$ is a nonzerodivisor on~$R$.
\end{lemm}

\begin{proof}
  As has already been remarked before, this follows from the flatness
  of~$R$ as an algebra over~$\ZZ[1/2d]$.
\end{proof}

Let us now reduce everything modulo~$p$. The reduction~$X/p$ of~$X$
modulo~$p$ is classified by a morphism
\begin{equation}\label{eq:classifying X/p}
	X/p\colon\Spec(R/p)\longrightarrow\M_{2d,p}.
\end{equation}

\begin{lemm}
  If the map~\eqref{eq:classifying X} classifying~$X$ is flat, so is
  the map~\eqref{eq:classifying X/p} classifying~$X/p$.
\end{lemm}

\begin{proof}
	Consider the following diagram.
   \begin{center}
  	\mbox{ 
	\xymatrix{
		\Spec(R/p)\ar[d]_{X/p}\ar[r] & \Spec(R)\ar[d]^X\\		
			\M_{2d,p}\ar[d]\ar[r] & \M_{2d}\ar[d]\\
		\Spec(\ZZ/p)\ar[r] & \Spec(\ZZ[1/2d])
    	} 
  }
	\end{center}
        The bottom square is a pullback by the definition
        of~$\M_{2d,p}$. The outer rectangle is a pullback by
        elementary algebra. It follows that the upper square is a
        pullback as well. The result follows by base change for flat
        morphisms.
\end{proof}

Recall from~\ref{subsec:Fermat} that the finite height
part~$\Mo_{2d,p}=\Mo_{2d,p,1}$ denotes the complement of the
supersingular locus~$\M_{2d,p,11}$ in the moduli
stack~$\M_{2d,p}=\M_{2d,p,1}$.

\begin{lemm}
  The map~\eqref{eq:classifying X/p} classifying~$X/p$ factors over
  the open substack~$\Mo_{2d,p}$.
\end{lemm}

\begin{proof}
  By assumption, the height of the closed fibre~$X/m$ of~$X/p$ is
  finite, and this height bounds the height of~$X/p$.
\end{proof}

Let us now see how the height stratification~\eqref{height
  stratification} from Section~\ref{subsec:Fermat} manifests in this
context. The pullbacks of the ideal sheaves which cut out the height
strata~$\M_{2d,p,h}$ in~$\M_{2d,p}=\M_{2d,p,1}$ along the
morphism~\eqref{eq:classifying X/p} give rise to a sequence of ideals
\begin{displaymath}
	0=J_1\subseteq J_2\subseteq J_3\subseteq\dots\subseteq R/p,
\end{displaymath}
such that~$I_{p,h}$ is the pre-image of~$J_h$ along~$R\rightarrow R/p$.

\begin{lemm}
	If the height of the closed fibre is~$h$, one has 
	$J_h\not=R/p$ and~\hbox{$J_{h+1}=R/p$},
	as well as
	$I_{p,h}\not=R$ and~$I_{p,h+1}=R$.
\end{lemm}

\begin{proof}
  The results for the~$J$ imply those for the~$I$, so we only need to
  prove the first ones.
	
  The ideal~$J_h$ cuts out the locus in~$\Spec(R/p)$ where the height
  is at least~$h$, and we have to show that this locus is not
  empty. But it contains the closed point.
	
  Similarly, the ideal~$J_{h+1}$ cuts out the locus in~$\Spec(R/p)$
  where the height is at least~\hbox{$h+1$}, and we have to show that
  this locus is empty. But it is closed and does not contain the
  closed point~$\Spec(R/m)$ by assumption. A geometric interpretation
  of Nakayama's Lemma gives the result: a closed subset of a spectrum
  of a noetherian local ring which does not contain the closed point
  is empty.
\end{proof}

The preceding lemma implies that~$v_{h}$ is a unit in~$R$, and, unless
the closed fibre is ordinary, we are left to deal with the
$v_1,\dots,v_{h-1}$.

\begin{lemm}
  The sequence~$(v_1,\dots,v_h)$ is regular on~$R/p$.
\end{lemm}

\begin{proof}
  This requires the more delicate results about the height stratification
  described in Section~\ref{subsec:Fermat}: the height
  stratification on~$\M_{2d,p}$ is defined locally by a regular
  sequence of sections of line bundles on the moduli stack. As flat
  morphisms preserve regularity, these sections pull back to a regular
  sequence on~$R/p$.
\end{proof}

As~$p$ is a nonzerodivisor on~$R$ by Lemma~\ref{lem:p}, it follows that 
the sequence
\begin{displaymath}
	(p,v_1,\dots,v_h)
\end{displaymath}
is regular on~$R$, with~$v_h$ a unit. This finishes the proof of
Theorem~\ref{thm:main}.

There are weaker versions of Theorem~\ref{thm:main} based on
corresponding versions of Landweber's result
modulo~$p$. (See~\cite{yagita} and~\cite{yosimura} for the latter.)
However, in order to impose ``brave new rings'' structures on the
resulting spectra, it is desirable to work with torsion-free
coefficient rings. The extra effort it took to achieve this will
pay off in the extra examples encompassed, to which we turn now.

\subsection{Examples}\label{subsec:final examples}

The rest of this section serves the purpose to show that all geometric
points of~$\M_{2d}$ of finite height can be thickened to give rise to
K3 spectra by means of the preceding Theorem~\ref{thm:main}. The
problem does not lie so much in finding a lifting to characteristic 0,
as there always is a lifting to the Witt ring, for example,
see~\cite{Ogus:Crystals},~\cite{Deligne:Relevements}. It lies in
finding such a lift with a flat map to the moduli stack. However, the
algebraicity of the stack provides such, as will now be explained.

For every degree~$2d$, there is a smooth surjection
\begin{displaymath}
	H\longrightarrow\M_{2d}
\end{displaymath}
where~$H$ is a suitable piece of a Hilbert scheme, 
see~\cite{Rizov:MixedCharacteristic} for example. Let~$R$ be one of the 
local rings of~$H$. As~$\M_{2d}$ has finite type over~$\ZZ[1/2d]$,
this will be noetherian. And if the residue field~$k$ of~$R$ has
characteristic prime to~$2d$, the ring~$R$ will be local
over~$\ZZ_{(p)}$. The composition
\begin{displaymath}
	\Spec(R)\longrightarrow H\longrightarrow\M_{2d}
\end{displaymath}
is flat as a composition of flat maps, and classifies a~K3 surface~$X$
over~$R$: the pullback of the universal family over the Hilbert
scheme. If the closed fibre of~$X$ over~$k$ has finite height,
Theorem~\ref{thm:main} applies to give an even periodic ring
spectrum~$E$ such that~\hbox{$\pi_0E\cong R$}
and~\hbox{$\Gamma_E\cong\fBr_X$}.

\begin{prop}
  Let~$R$ be one of the local rings of the Hilbert scheme covering~$
  \M_{2d}$ with residue field of characteristic prime to~$2d$. Then
  there is an even periodic ring spectrum~$E$ such
  that~\hbox{$\pi_0E\cong R$} and~\hbox{$\Gamma_E$} is isomorphic to
  the formal Brauer group of the germ of the universal family over~$R$.
\end{prop}

The reader may wonder why a smooth cover has been used in the
preceding discussion, while an etale cover is available for the
Deligne-Mumford stack~$\M_{2d}$. The reason is that the smooth cover
can be made fairly concrete using the Hilbert schemes above, whereas
the exis\-tence of an etale cover is only guaranteed by means of an
unramified diagonal, which implies for abstract reasons the existence
of etale slices for the smooth cover, see (4.21) in~\cite{DM} or (8.1)
in~\cite{LM-B}. This is in contrast to the case of elliptic curves,
where etale covers can be written down explicitly.

\begin{prop}\label{prop:local}
    If~$X\colon\Spec(k)\rightarrow\M_{2d}$ is a geometric point of
    finite height and characteristic prime to~$2d$, there is an even
    periodic ring spectrum~$E$ such that~\hbox{$\pi_0E$} is isomorphic
    to the local ring of~$\M_{2d}$ in~$X$ (with respect to the etale
    topology) and the reduction of~\hbox{$\Gamma_E$} to~$k$ is the
    formal Brauer group of~$X$.
\end{prop}

\begin{proof}
  The local ring~$\widetilde\O_{\M_{2d},X}$ (with a tilde to indicate
  the etale topology) is the colimit of the local rings~$\O_{U,u}$,
  where~$(U,u)$ runs over the etale neighbourhoods
	\begin{center}
  	\mbox{ 
		\xymatrix{
		\Spec(k)\ar[dr]_u\ar[rr]^X && \M_{2d}\\
			&U\ar[ur]_{\mathrm{etale}}&
    		} 
  	}
	\end{center}
	of~$X$ in~$\M_{2d}$.  As in the proof of the previous
        proposition, the hypotheses of Theorem~\ref{thm:main} are
        satisfied for the local rings~$\O_{U,u}$.
        As~$\widetilde\O_{\M_{2d},X}$ is the (strict) henselisation of
        the~$\O_{U,u}$, it is flat over them. Therefore, the
        hypotheses of Theorem~\ref{thm:main} are satisfied
        for~$\widetilde\O_{\M_{2d},X}$ as well.
\end{proof}

There is a weaker variant of the preceding proposition which replaces
the henselian rings by complete rings. To state it,
let~$\widehat\O_{\M_{2d},X}$ be the completion of the local
ring~$\widetilde\O_{\M_{2d},X}$. Its (formal) spectrum is, by
definition, the formal deformation space of~$\M_{2d}$ at~$X$, so that,
conversely,~$\widehat\O_{\M_{2d},X}$ is the ring of formal functions
on it.

\begin{prop}\label{prop:formal}
  If~$X\colon\Spec(k)\rightarrow\M_{2d}$ is a geometric point of
  finite height and characteristic prime to~$2d$, there is an even
  periodic ring spectrum~$E$ such that~\hbox{$\pi_0E$} is isomorphic
  to the ring of formal functions on the formal deformation space
  of~$\M_{2d}$ in~$X$ and such that the reduction of~\hbox{$\Gamma_E$} to~$k$
  is the formal Brauer group of~$X$.
\end{prop}

\begin{proof}
  As all rings involved are noetherian, the
  completion~$\widetilde\O_{\M_{2d},X}\rightarrow\widehat\O_{\M_{2d},X}$
  is flat and one may argue as before.
\end{proof}

%%%%%%%%%%%%%%%%%%%%%%%%%%%%%%%%%%%%%%%%%%%%%%%%%%%%%%%%%%%%%%%%%%%%%

\section*{Acknowledgements}

I would like to thank those who have discussed this material with me
before, especially Mike Hopkins and Jack Morava; my debts to them are
clear.

%%%%%%%%%%%%%%%%%%%%%%%%%%%%%%%%%%%%%%%%%%%%%%%%%%%%%%%%%%%%%%%%%%%%%

%%%%%%%%%%%%%%%%%%%%%%%%%%%%%%%%%%%%%%%%%%%%%%%%%%%%%%%%%%%%%%%%%%%%%

\vfill

\parbox{\linewidth}{%
Markus Szymik\\
Department of Mathematical Sciences\\
NTNU Norwegian University of Science and Technology\\
7491 Trondheim\\
NORWAY\\
\href{mailto:markus.szymik@ntnu.no}{markus.szymik@ntnu.no}\\
\href{https://folk.ntnu.no/markussz}{folk.ntnu.no/markussz}}


\begin{thebibliography}{99}

\bibitem{Artin} {M. Artin}, Supersingular~K3 surfaces,
  Ann. Scient. Ecole Norm. Sup. 7 (1974) 543--658.

\bibitem{Artin+Mazur} {M. Artin, B. Mazur}, Formal groups
  arising from algebraic varieties, Ann. Scient. Ecole Norm. Sup. 10
  (1977) 87--132.
  
\bibitem{Badescu} {L. B\u{a}descu}, Algebraic surfaces,
  Springer-Verlag, New York, 2001.
   
\bibitem{Beauville} {A. Beauville}, Complex algebraic
  surfaces, London Mathematical Society Student Texts 34, Cambridge
  University Press, Cambridge, 1996.

\bibitem{Behrens+Lawson} {M. Behrens, T. Lawson}, Topological
  automorphic forms, Preprint.

\bibitem{Deligne:Relevements} {P. Deligne}, Rel\`{e}vement des
  surfaces~K3 en caracteristique nulle, prepared for publication by
  Luc Illusie, 58--79 in Algebraic surfaces (Orsay 1976-78), Springer,
  1981.

\bibitem{DM} {P. Deligne, D. Mumford}, 
The irreducibility of the space of curves of given genus, 
Inst. Hautes \'Etudes Sci. Publ. Math. No. 36 (1969) 75--109. 

\bibitem{Devoto} {J.A. Devoto}, Quaternionic elliptic objects
  and~$K3$-cohomology, 26--43 in Elliptic cohomology, London
  Math. Soc. Lecture Note Ser. 342, Cambridge Univ. Press, Cambridge,
  2007.
  
\bibitem{Fröhlich} {A. Fr\"ohlich}, Formal groups,
Lecture Notes in Mathematics 74, Springer-Verlag, Berlin-New York, 1968.

\bibitem{vanderGeerKatsura:Stratification} {G. van der Geer,
  T. Katsura}, On a stratification of the moduli of~K3 surfaces,
  J. Eur. Math. Soc. 2 (2000) 259--290.

\bibitem{vanderGeerKatsura:Heights} {G. van der Geer,
    T. Katsura}, On the height of Calabi-Yau varieties in positive
  characteristic, Doc. Math. 8 (2003) 97--113.
		
\bibitem{GoerssHopkins:Summary} {P.G. Goerss, M.J. Hopkins},
  Moduli spaces of commutative ring spectra, 151--200 in Structured
  ring spectra, London Math. Soc. Lecture Note Ser. 315, Cambridge
  Univ. Press, 2004.

\bibitem{Gorbounov+Mahowald} {V. Gorbounov, M. Mahowald}, Formal
  completion of the Jacobians of plane curves and higher
  real~$K$-theories, J. Pure Appl. Algebra 145 (2000) 293--308.

\bibitem{Goto} {Y. Goto}, A note on the height of the formal
  Brauer group of a~K3 surface, Canad. Math. Bull. 47 (2004) 22--29.

\bibitem{Hirokado} {M. Hirokado}, A non-liftable Calabi-Yau
  threefold in characteristic 3, Tohoku Math. J. 51 (1999) 479--487.

\bibitem{Hopkins:ICM1994} {M.J. Hopkins}, Topological modular
  forms, the Witten genus, and the theorem of the cube, 554--565 in
  Proceedings of the International Congress of Mathematicians 1994,
  Vol. I, Birkh\"auser, 1995.

\bibitem{Landweber:Exactness} {P.S. Landweber}, Homological
  properties of comodules over~$\mathrm{MU}_*(\mathrm{MU})$
  and~$\mathrm{BP}_*(\mathrm{BP})$, American Journal of Mathematics 98
  (1976) 591--610.

\bibitem{LRS} {P.S. Landweber, D.C. Ravenel,, R.E. Stong},
  Periodic cohomology theories defined by elliptic curves,
  317--337 in The \v Cech centennial (Boston 1993), Contemp. Math. 181,
  Amer. Math. Soc., Providence, RI, 1995.

\bibitem{LM-B} {G. Laumon, L. Moret-Bailly}, 
Champs alg\'ebriques, Springer-Verlag, Berlin, 2000.

\bibitem{Lazard:Classification} {M. Lazard}, Sur les groupes
  de Lie formels \`a un param\`etre, Bull. Soc. Math. France 83 (1955)
  251--274.

\bibitem{Lazard:LectureNotes} {M. Lazard}, 
Commutative formal groups, 
Lecture Notes in Mathematics 443, Springer-Verlag, Berlin-New York, 1975.

\bibitem{Lurie:Elliptic} {J. Lurie}, A Survey of Elliptic
  Cohomology, Preprint.
  
\bibitem{MMSS} {M.A.Mandell, J.P. May, S. Schwede,,
    B. Shipley}, Model categories of diagram spectra, Proc. London
  Math. Soc. 82 (2001) 441--512.

\bibitem{Miller:Landweber} {H. Miller}, Sheaves, gradings, and the
  exact functor theorem, Preprint. % 2003
		
\bibitem{Morava} {J. Morava}, Report on Mike Hopkins' work on
  Calabi-Yau cohomology, Midwest Topology Seminar, Spring 1992.

\bibitem{Mukai} {S. Mukai}, Curves,~K3 surfaces and Fano 3-folds
  of genus~$\leqslant10$, 357--377 in Algebraic Geometry and Commutative
  Algebra I, Kinokuniya, Tokyo, 1988.

\bibitem{Ogus:Crystals} {A. Ogus}, Supersingular~K3 crystals,
  3--86 in Journ\'ees de G\'eom\'etrie Alg\'e\-brique de Rennes (Rennes
  1978), Vol. II, Ast\'erisque, 64, Soc. Math. France, 1979.

\bibitem{Ogus:HeightStrata} {A. Ogus}, Singularities of the height
  strata in the moduli of~K3 surfaces, 325--343 in Moduli of Abelian
  Varieties (Texel Island 1999), Progr. in Math. 195, Birkh\"auser, 2001.

\bibitem{Ravenel} {D.C. Ravenel}, Toward higher chromatic analogs
  of elliptic cohomology, 286--305 in Elliptic cohomology, London
  Math. Soc. Lecture Note Ser. 342, Cambridge Univ. Press, Cambridge,
  2007.

\bibitem{Rezk:HopkinsMiller} {C. Rezk}, Notes on the
  Hopkins-Miller theorem, 313--366 in Homotopy theory via algebraic
  geometry and group representations (Evanston 1997),
  Contemp. Math. 220, Amer. Math. Soc., 1998.

\bibitem{Rizov:MixedCharacteristic} {J. Rizov}, Moduli stacks of
  polarized~K3 surfaces in mixed characteristic, Serdica Math. J. 32
  (2006) 131--178.

\bibitem{Stienstra:fgl} {J. Stienstra}, Formal group laws
  arising from algebraic varieties, Amer. J.  Math. 109 (1987)
  907--925.
  
\bibitem{Szymik} M. Szymik. Crystals and derived local moduli for ordinary~K3 surfaces. Adv. Math. 228 (2011) 1--21.

\bibitem{Thomas:Survey} {C.B. Thomas}, Elliptic cohomology,
  409--439 in Surveys on surgery theory, Vol. 1, Ann. of Math. Stud.
  145, Princeton Univ. Press, 2000.

\bibitem{Thomas:Book} {C.B. Thomas}, Elliptic cohomology, Kluwer
  Academic/Plenum Publishers, 1999.

\bibitem{yagita} {N. Yagita}, The exact functor theorem
  for~$BP_*/I_n$-theory, Proc. Japan Acad 52 (1976) 1--3.

\bibitem{yosimura} {Z. Yosimura}, Projective dimension of
  Brown-Peterson homology with modulo~$(p,v_1,\dots,v_{n-1})$
  coefficients, Osaka J. Math. 13 (1976) 289--309.

\bibitem{Yui} {N. Yui}, Formal Brauer groups arising from certain
  weighted~K3 surfaces, J. Pure Appl. Algebra 142 (1999) 271--296.

\end{thebibliography}
\end{document}